\documentclass{article}

\usepackage{amsmath,amssymb,mathtools,amsthm}
\usepackage{xcolor,enumitem}

\usepackage[a4paper]{geometry}

\usepackage{cancel}

\def\elle#1{L^{#1}(\Omega)}

\def\w{W_0^{1,2}(\Omega)}
\def\wm{W_0^{1,m^{*}}(\Omega)}

\def\io{\int_\Omega}

\def\norma#1#2{\|#1\|_{\lower 4pt \hbox{$\scriptstyle #2$}}}

\def\un{u_n}

 \newcount\lavn
 \newcount\rasn
 \newcount\tavn

 \def\nn{{\global\advance\lavn by 1}C_{ \the\lavn }}

\newcommand{\Rd}{{\mathbb R^d}}

\DeclareMathOperator{\diver}{div}
\DeclareMathOperator{\dist}{dist}

\DeclareMathOperator{\sign}{sign}

\usepackage[hidelinks]{hyperref}
\usepackage[nameinlink]{cleveref}

\newtheorem{theorem}{Theorem}%
\newtheorem{proposition}[theorem]{Proposition}%
\newtheorem{corollary}[theorem]{Corollary}%
\newtheorem{lemma}[theorem]{Lemma}%

\theoremstyle{definition}
\newtheorem{remark}[theorem]{Remark}%
\newtheorem*{remark*}{Remark}%

\begin{document}

\title{Failure of the {Hopf-Oleinik lemma} for {a} linear elliptic problem with singular convection of non-negative divergence}
\author{
	L. Boccardo%
	\thanks{Istituto Lombardo \& Sapienza Universit\`a di Roma. \texttt{boccardo@mat.uniroma1.it}}
	\and  
	J. I. D\'iaz%
	\thanks{Inst. de Matemática Interdisciplinar \& Fac. de Matemáticas, U. Complutense de Madrid. \texttt{jidiaz@ucm.es}}
	\and 	
	D. G\'omez-Castro%
	\thanks{Inst. de Matemática Interdisciplinar \& Fac. de Matemáticas, U. Complutense de Madrid. \texttt{dgcastro@ucm.es}}
} 

\maketitle 

\begin{abstract}
In this paper we study existence, uniqueness, and integrability of solutions to the Dirichlet problem
$-\mathrm{div}( M(x) \nabla u ) = -\mathrm{div} (E(x) u) + f$ in a bounded domain of $\mathbb R^N$ with $N \ge 3$. We are particularly interested in singular $E$ with $\mathrm{div} E \ge 0$. We start by recalling known existence results when $|E| \in L^N$ that do not rely on the sign of $\mathrm{div} E $. Then, under the assumption that $\mathrm{div} E \ge 0$ distributionally, we extend the existence theory to $|E| \in L^2$. For the uniqueness, we prove a comparison principle in this setting. Lastly, we discuss the particular cases of $E$ singular at one point as $Ax /|x|^2$, or towards the boundary as $\mathrm{div} E \sim \mathrm{dist}(x, \partial \Omega)^{-2-\alpha}$. In these cases the singularity of $E$ leads to $u$ vanishing to a certain order.
In particular, this shows that the {Hopf-Oleinik lemma, i.e. $\partial u / \partial n < 0$,} 
fails in the presence of such singular drift terms $E$.
\end{abstract}

\section{Introduction}

It is well known that many relevant applications lead to the presence of a convection term in the correspondent model which, in its simplest formulation, leads to a boundary value problem for linear elliptic second order equation of the following type:
\begin{equation}
	\label{eq:PDE}
	\begin{dcases}
		-\diver (M(x) \nabla u) = - \diver(u E(x)) + f(x) & \text{in } \Omega \\
		u = 0 & \text{on } \partial \Omega.
	\end{dcases}
\end{equation}
Here $\Omega \subset \mathbb R^N$,
$N\ge 3$,
is an open, bounded set, and we assume that $M \in L^\infty(\Omega)^{N\times N}$ is elliptic 
\begin{equation*} 
	M (x)\xi \cdot \xi \ge \alpha |\xi|^2	, \qquad \forall \xi \in \mathbb R^N \text{ and a.e.\ } x \in \Omega.
\end{equation*}
According to the regularity of the right-hand side datum $f(x)$ it is natural to search the solution in the energy space $W^{1,2}_0 (\Omega)$ (case of $f\in H^{-1}(\Omega )$: see, e.g. \cite{Stampacc,Gilbarg-Trudinger,Boccardo2009}), or in a larger Sobolev space if $f$ is singular (see \cite{Boccardo2009});  when $f\in L^{1}(\Omega )$, see, for instance, \cite{Brezis-Strauss}, or when $L^{1}(\Omega ,\delta )$ with $\delta (x)=d(x,\partial \Omega )$, see, e.g., 
\cite{Brezis-Cazenave et al,diaz-gc-rakotoson-temam2018}.

In the mentioned references it assumed that the convection term is regular (for instance  $E\in W^{1,\infty }(\Omega )$) and that it satisfies an additional condition which helps  to have a maximum principle: 
\begin{equation}
	\label{eq:diver E non-negative}
\diver E\geq 0\text{ a.e. on }\Omega .
\end{equation}

More recently, some effort has been devoted to get an existence and regularity theory under more general conditions on the convection term $E$ by different authors (see, e.g. \cite{Boccardo2009}, \cite{Boccardo+Orsina2019} and their references). For instance, solutions in the energy space can be considered under the conditions $|E| \in L^N (\Omega)$ and $f \in L^{\frac{2N}{N+2}} (\Omega)$. 
In \cite{diaz-gc-rakotoson-temam2018} and \cite{diaz+gc+rakotoson2017schrodinger}
the authors study
the case in which $|E| \in L^N (\Omega)$ and $\diver E=0$ in $\Omega$ and $E\cdot n=0$ on $\partial \Omega $,  $f\in L^{1}(\Omega ,\delta )$.
{See also \cite{Droniou2002,Kim2020}.} 

\bigskip 

In this paper, we will show that \eqref{eq:diver E non-negative} makes $\diver E$ behave like a non-negative potential in the Schrödinger case, and we can apply techniques from that setting. See, for example,  \cite{diaz+gc+rakotoson2017schrodinger,diaz-gc-rakotoson-temam2018,diaz-gc-vazquez2018,GC+Vazquez2018}.
We focus on the case where \eqref{eq:diver E non-negative} holds in distributional sense.

The paper is structured as follows. First, in \Cref{sec:E in L^N} we review known results for the case $|E| \in L^N$ and $f \in L^{\frac{2N}{N+2}} (\Omega)$ which were published in \cite{Boccardo2009}, were shown there is a unique weak solution of \eqref{eq:PDE} that can be constructed by approximation. In \Cref{sec:E in L^2} we show that if $|E| \in L^2 (\Omega)$, $\diver E \ge 0$, and $f \in L^m (\Omega)$ for some $m > 1$ then the same approximation procedure converges to a weak solution of \eqref{eq:PDE}, and we give some a priori bounds for this solution. In \Cref{sec:comparison principle} we show that, if we also assume $f \in L^{\frac{2N}{N+2}} (\Omega)$, then this constructed solution is the unique weak solution of \eqref{eq:PDE}.

\bigskip 

Then we move to discussing interesting examples that fall in this setting. 
In \Cref{sec:convection singular one point} we focus on the case 
\begin{equation}
	E (x) =	A \frac{x}{|x|^2}, 	
\end{equation}
which is somehow in the limit of theory since it is not in $L^N(\Omega)$ but it is in $L^{r} (\Omega)$ for $r \in [1,N)$. 
In \cite{Boccardo+Orsina2019} the authors examined the more general class 
\begin{equation} 
	|E|\le\frac{|A|}{|x|}.
\end{equation}
The authors show existence of solutions $u$, where the summability is reduced as $|A|$ is increased.
Their results indicate that the sign of $A$ should play a role, but the application of Hardy's inequality (which they use in a crucial way) is not able to detect this fact. 
In \Cref{prop:convection at one point} we show that if $N > 1$, $f \in L^m (\Omega)$ for suitable $m$, and $A > 0$ then we can use the sign of $\diver E$ to deduce that
{the solution $u_A$ of \eqref{eq:PDE} with $E = A \frac{x}{|x|^2}$ satisfies}
\begin{equation*}
	u_A \to 0 \text{ in } L^1 (\Omega) \text{ as } A \to +\infty.
\end{equation*}
By the contrary, when $A < 0$ we cannot improve the result in \cite{Boccardo+Orsina2019}. 
Notice that this is similar to the equation $L(u_B)+B\,u_B=f$, whereas $B\to\infty$ we have  $u_B\to0$.

\bigskip 

Lastly, in \Cref{sec:convection singular boundary}, we discuss the case where $E$ is 
{suitably}
singular only on the boundary. We present an example showing that if $\diver E$ behaves like 
{
	$d(x,\partial \Omega )^{-2-\gamma}$} 
for some $\gamma > 0$ and $f$ is bounded, then the solutions are flat on the boundary, i.e.
$$		
	|u(x)| \le C \dist(x, \partial \Omega)^\alpha \text{ for some } \alpha > 1.  
$$
{
In particular, this shows that the {Hopf-Oleinik lemma}, i.e. $\frac{\partial u}{\partial n} < 0$ on $\partial \Omega$, fails in the presence of such singular drift terms $E$.
Our example can be easily extended to a more general class of $E$, as we comment in \Cref{sec:extension}.}
Again, we use the fact that $\diver E$ acts as a potential. However, in the Schrödinger equation {it} is sufficient that $V(x) \ge C \delta^{-2}$ to get flat solutions, whereas for $E$ we need a strictly larger exponent (see \Cref{rem:19 de Ildefonso}).
Questions of this
type are quite relevant in the framework of linear Schrödinger equations
associated to singular potential since they can be understood as complements
to the Heisenberg Incertitude Principle (see, e.g. \cite{Diaz (sema1),Diaz (sema2),diaz+gc+rakotoson2017schrodinger,diaz-gc-rakotoson-temam2018,OrPonce,diaz-gc-vazquez2018}). 

\bigskip 

We conclude with some further comments and open problems in \Cref{sec:extension}.

\section{Known results when \texorpdfstring{$|E| \in L^N$}{|E| in LN}}
\label{sec:E in L^N}

We define the Sobolev conjugate exponent
\begin{equation*}
	m^* = \frac{mN}{N-m} \quad \text{if } m \le N, 
	\qquad \text{ and } \qquad  
	m^{**} = (m^*)^* 
	=\frac{mN}{N-2m} \quad \text{if } m \le \frac N 2.
\end{equation*}
We have that $m^{**} \in [1,\infty]$ for 
$\frac{N}{N+2} \le m \le \frac{N}{2}$. 
Notice that $m^* \ge 2$ if and only if $m \ge \frac{2N}{N+2} = (2^*)'$.
Notice that, since $m \ge 1$ we have $m^* \ge m$. In order to compute explicit a priori estimates, we use the Sobolev embedding constant $S_p$ such that, for $1 < p <+\infty$
\begin{equation}
	\label{eq:Sobolev embedding}
	S_p \| u \|_{L^{p^*} (\Omega)} \le \| \nabla u \|_{L^p(\Omega)}.
\end{equation}
We point out the relevance of the constants, for $N > 2$ of
$
(2^*)' = \frac{2N}{N+2} .
$
This constant {depends only of $N$}.
Since we are going to require the Sobolev embedding for $p =2$, we assume that $N \ge 3$. 
In \cite{Boccardo2009} the author proves the following existence theorem with a priori estimates. 
\begin{theorem}[\cite{Boccardo2009}]
	\label{thm:Boccardo well-posedness}
	Let $f \in L^{\frac{2N}{N+2}} (\Omega)$ and $|E| \in L^N(\Omega)$. Then, there exists a unique weak solution $u$ {of \eqref{eq:PDE}} in the sense that
	\begin{equation*}
		u \in W^{1,2} _0(\Omega) \text{ is such that } \io M(x)\nabla u \nabla v
		= \io u\,E(x)\cdot\nabla v
		+\io  f (x)\,v(x), \; 
		\forall\;v\in W^{1,2}_0 (\Omega).
	\end{equation*}
	and it satisfies:
	\begin{enumerate}
		\item Logarithmic estimate:
		\begin{equation*}
			\left( \int_\Omega |\log (1 + |u|)|^2  \right)^{\frac {2}{2^*}} \le \frac{1}{S_2^2 \alpha^2} \int_\Omega|E|^2 + \frac{2}{S_2^2 \alpha} \int_\Omega |f|,
		\end{equation*}
		\item Gradient estimate: there exists $C = C(\alpha, N)$ such that
		\begin{equation}
			\label{eq:Theorem 1 gradient estimate}
			\int_\Omega |\nabla u|^2 
			\le C  \left(
			\| E \|_{L^N}^2   + \| f \|_{L^{\frac{2N}{N+2}}}^2 \right).
		\end{equation}
		\item 
		Stampacchia-type summability: For $m \in [ \frac{2N}{N+2} , \frac N 2 )$  there exists a constant $C = C( m , \alpha,$ $N, \| E \|_{L^N})$ such that
		\begin{equation}
			\label{eq:Lm** estimate}
			\| u \|_{m^{**}} \le C \| f \|_{m}.
		\end{equation}
		\item 
		Stampacchia-type boudedness:
		Let $r > N$ and $m > \frac N 2$. There exists $C$ such that
		\begin{equation}
			\label{eq:Linfty estimate}
			\| u \|_{L^\infty} \le C( m,r, \alpha, \| f \|_{L^m}, \| E \|_{L^r} ).
		\end{equation}
	\end{enumerate}
\end{theorem}
\begin{remark}
	The natural theory for this problem in energy space is precisely $|E| \in L^N (\Omega)$, since in the weak formulation we need to justify a term of the form $E u \nabla v$, where $u, v \in W_0^{1,2} (\Omega)$. This means that $u \in L^{2^*}$ whereas $\nabla v \in L^2$. 
	So we always have that $u E \in L^2 (\Omega)$.
\end{remark}

In  \cite{Boccardo2009} the main tool to study the linear problem \eqref{eq:PDE} are the auxiliary non-linear Dirichlet problems 
\begin{equation}
	\label{pn}
	\begin{dcases}
		-\diver (M(x) \nabla u_n) = - \diver\left( \frac{u_n}{1+ \frac 1 n u_n } E_n(x)\right) + f_n(x) & \Omega \\
		u = 0 & \partial \Omega,
	\end{dcases}
\end{equation}
where the author take $f_n = T_n (f)$ a truncation of $f$ through the family
\begin{equation*}
	T_n (s) = \begin{dcases}
		s & |s| \le k , \\
		k \sign(s) & |s| > k,
	\end{dcases}
\end{equation*}
and $E_n = \frac{E}{1+\frac 1 n |E|}$. We will take advantage of a similar approximation.

\begin{remark}
Since the problem is linear, for $t \in \mathbb R$ we have that $t u$ is solution of 
$$-\diver (M(x) \nabla [t\,u]) = - \diver([t\,u] E(x)) + t\,f(x), $$ 
and that $E$ does not change. Thus, {using \eqref{eq:Theorem 1 gradient estimate}}
$$
t^2\int_\Omega |\nabla u|^2 
\le C  \left(
\| E \|_{L^N}^2 
+t^2 \| f \|_{L^{\frac{2N}{N+2}}}^2 \right).
$$
Dividing by $t^{-2}$ and taking the limit as $t\to\infty$ gives	
\begin{equation} 
	\label{eq:H1 estimate}
	\int_\Omega |\nabla u|^2 
	\le C     \| f \|_{L^{\frac{2N}{N+2}}}^2.
\end{equation}
Notice that in \Cref{thm:uniqueness} we will prove this fact for the case $\diver E \ge 0$.
\end{remark}

\section{Existence theory when \texorpdfstring{$|E| \in L^2$}{|E| in L2}  and \texorpdfstring{$\diver E \ge 0$}{divE ge 0}}
\label{sec:E in L^2}

The structural assumption in this section is the following:
\begin{equation}
	\label{e}
	\begin{dcases}
		E \,\hbox{   belongs to the  Lebesgue space } \;(\elle2)^N,
		\\
		\diver E \ge 0 \text{ in } \mathcal D' (\Omega), \;\hbox{ that is }\,\io E\cdot \nabla \phi \leq 0,\;\forall\;0\leq\phi\in\w .
	\end{dcases}
\end{equation}

\begin{theorem}
	\label{thm:existence E in L^2}
	Assume \eqref{e} and
	\begin{equation}
		\label{f}
		f\in\elle m, \;1<m<\frac N2,
	\end{equation}
	and let $p = \min\{ 2 , m^* \}$.  Then, there exists a weak solution $u$
	{of \eqref{eq:PDE}}
	in the sense that 
	\begin{equation}
		\label{eq:weak solution}
		u \in W^{1,p} _0(\Omega) \text{ is such that } \io M(x)\nabla u \nabla v
		= \io u\,E(x)\cdot\nabla v
		+\io  f (x)\,v(x), \; 
		\forall\;v\in W^{1,\infty}_0 (\Omega).
	\end{equation}
	Furthermore, it satisfies
	\begin{equation}
		\label{est2}
		\begin{cases}
			\norma{u}{\wm}
			\leq C_m\,
			\norma{f}{m},&\hbox{if }1<m<\frac{2N}{N+2};
			\\
			\norma{u}{\w}+\norma{u}{m^{**}}
			\leq \tilde{C_m}\,
			\norma{f}m ,&\hbox{if }\frac{2N}{N+2}\leq m {<} \frac N2.
		\end{cases}
	\end{equation}
\end{theorem} 
\begin{remark}
	Due to the gradient estimates, we can extend \eqref{eq:weak solution} to all $v \in W^{1,q}_0 (\Omega)$ by approximation, 
	where $q=p'$.
\end{remark}

Since the construction of solutions in the proof of \Cref{thm:existence E in L^2}	is achieved by approximation, we have that
\begin{corollary}
	The solutions constructed in \Cref{thm:existence E in L^2}	satisfy \eqref{eq:Lm** estimate} and \eqref{eq:Linfty estimate}.
\end{corollary}

We say that $\un$ {is a} weak solution of \eqref{pn} if $u \in W_0^{1,2} (\Omega)$ is such that 
\begin{equation}
\label{wpn}
\io M(x)\nabla u_n \nabla v
= \io\frac{u_n}{1+\frac1n|u_n|}\,E_n(x)\cdot\nabla v
+\io  f_n (x)\,v(x),\quad
\forall\;v\in\w.
\end{equation}
The existence of a weak solution if $E_n \in L^2 (\Omega)^N$
is a consequence of the Schauder theorem. 
The proof of \Cref{thm:existence E in L^2} is based on the following approximation lemma  
\begin{lemma}
{Let $u_n$ be any weak solution of \eqref{wpn} with $E_n = E$},  \eqref{e}, \eqref{f}, and $f_n = T_n (f)$.  Then, for any weak solution $u_n$ of \eqref{wpn} we have that
\begin{equation}
	\label{est2b}
	\begin{cases}
		\norma{\un}{\wm}
		\leq C_m\,
		\norma{f}{m},&\hbox{if }1<m<\frac{2N}{N+2};
		\\
		\norma{\un}{\w}+\norma{\un}{m^{**}}
		\leq \tilde{C_m}\,
		\norma{f}m,&\hbox{if }\frac{2N}{N+2}\leq m\leq \frac N2.
	\end{cases}
\end{equation} 
where
\begin{equation}
	\label{20jenero}
	\text{$C_m$ does not depend on $E$}.
\end{equation}
Hence, up to a subsequence, $\{\un\}$ converges weakly in $L^{m^{**}}$.
\end{lemma}

\begin{proof} 
	Our proof is the same of \cite{bg2}, since we will see that the contribution of new term on $E$ is a negative number.
	We use ${T_k(\un)}|{T_k(\un)}|^{2\gamma -2}$ as test function in \eqref{wpn}, 
	$\gamma=\frac{m^{**}}{2^*}$; we repeat it is possible since every ${T_k(\un)}$ has exponential summability.
	Note that $2\gamma-1>0$ since $m>1$.
	Thus, we have
	\begin{align*}
	\io M(x)\nabla u_n \nabla ({T_k(\un)}|{T_k(\un)}|^{2\gamma -2})
	&= \io\frac{u_n}{1+\frac1n|u_n|}\,E(x)\cdot\nabla({T_k(\un)}|{T_k(\un)}|^{2\gamma-2}) 
	\\
	&\qquad +\io  f_n (x)\, {T_k(\un)}|{T_k(\un)}|^{2\gamma -2}.
	\end{align*}
	To study the second integral, we define the function
	$$
	H_\gamma(s)
	=\int_0^s\frac{t\,|t|^{2\gamma-2}}{1+\frac1n|t|}\,dt. 
	$$
	It is easy to check that $H_\gamma (s) \ge 0$ for all $s\in \mathbb R$. Thus, using the sign condition on $\diver E$ we have that
	\begin{align*} 
	\io\frac{u_n}{1+\frac1n|u_n|}\,& E(x)\cdot\nabla({T_k(\un)}|{T_k(\un)}|^{2\gamma-2})
	\\
	&=
	\io(2\gamma-1)
	\frac{T_k(\un)  \,|{T_k(\un)}|^{2\gamma-2}}{1+\frac1n|T_k(\un)|}\,E(x)\cdot\nabla{T_k(\un)}
	\\
	&
	=\io {{H}}_\gamma({T_k(\un)})\,E(x)\cdot\nabla{T_k(\un)}
	=\io \,E(x)\cdot\nabla[H_\gamma({T_k(\un)})]\leq0.
	\end{align*} 
	Hence, we have that
	$$
	\io M(x)\nabla u_n \nabla ({T_k(\un)}|{T_k(\un)}|^{2\gamma -2})
	\leq
	 \io  f_n (x)\, {T_k(\un)}|{T_k(\un)}|^{2\gamma -2},
	$$
	which is the starting point of \cite{bg2}, and we get the estimates
	$$
	\begin{cases}
	 \norma{{T_k(\un)}}{\wm}
	\leq C_m\,
	\norma{f}{m},\quad\hbox{if }1<m<\frac{2N}{N+2};
	\\
	 \norma{{T_k(\un)}}{\w}+\norma{{T_k(\un)}}{m^{**}}
	\leq \tilde{C_m}\,
	\norma{f},\quad\hbox{if }\frac{2N}{N+2}\leq m{<} \frac N2.
	\end{cases}
	$$
	Letting $k\to\infty$ we recover \eqref{est2b}.
	\end{proof}

With this lemma, we can pass to the limit to prove \Cref{thm:existence E in L^2}.

\begin{proof}[Proof of \Cref{thm:existence E in L^2}]
	Up to subsequences,  the sequence $\{{T_k(\un)}\}$ constructed above,
	weakly converges (in $\wm$ or  in $\w\cap\elle{m^{**}}$) and it is possible to pass to some $u$ (note that $u\in\elle{m^{**}}$).
	Recall that  $E\in (L^2)^N $. 
	In order to pass to the limit in
	$$\io\frac{u_n}{1+\frac1n|u_n|}\,E_n(x)\cdot\nabla v$$
	in \eqref{wpn} 
	we need
	$$
	1\geq\frac1{m^{**}}+\frac1N+\frac12.
	$$
	That is equivalent to $m\geq\frac{2N}{N+2}$.
	Thus we pass also to the limit in \eqref{est2b}.
\end{proof}

\begin{remark}
	Note that, once more it is possible to develop an approximate method in order to prove the existence when $E \in L^r$.
	Indeed, 
	let $E_0\in L^r$, $r>1$ and $E_n\in L^2$ converging to $E_0$ in $L^r$. 
	Define now $\un$ in the corresponding way, we can
	use the statement of \eqref{20jenero}, so that we can say that
	estimates \eqref{est2} still hold for this new sequence $\{\un\}$
	and once more we can pass to the limit, and we prove the existence if
	$$
	1\geq\frac1r+\frac1m-\frac2N
	$$
\end{remark}

We can provide further a priori estimates when $ \diver E \ge 0$
\begin{proposition}
	\label{thm:a priori diver E non-neg}
	The solutions constructed in \Cref{thm:existence E in L^2} satisfy the following additional estimates:

	\begin{enumerate}
		\item \label{it:sum u in L1}
		($L^1$ estimate) 
		If $\diver E \in L^1 (\Omega)$ then 
		we have that
		\begin{equation}
			\label{eq:uniform bound 2}
			\int_\Omega |u| \diver E \le \int_\Omega |f|. 
		\end{equation} 
		
		\item \label{it:sum u in m greater 1}
		($L^{{m}}$ estimate) If $\diver E \ge c_0 > 0$ and $m > 1$ then
		\begin{equation}
			\label{eq:uniform estimate approximation 3}
			\| u \|_{L^{{m}}} \le \frac{{{m}}}{{{m}}-1} \frac{\| f \|_{L^m} }{c_0}.
		\end{equation} 
	\end{enumerate}
\end{proposition}

{
We will later take advantage of \eqref{eq:uniform bound 2} and present several extensions. See, e.g., \Cref{thm:existence singular boundary} where we extend the result to $\diver E \in L^1_{loc}$.
}

\begin{remark}
	Notice that \eqref{eq:uniform estimate approximation 3} blows up as $m \to 1$. 
	In fact, it is known that the case $m = 1$ does not satisfy such an estimate.
\end{remark}

We prove a priori estimates under the assumption of {$\diver E \ge 0$} for bounded (or even smooth) $E$, which we now know will hold for approximations.

\begin{proof}[Proof of \Cref{thm:a priori diver E non-neg}]
	Assume first that $E \in (L^N)^N$, and $f \in L^m$ for $m \ge \frac{2N}{N+2}$. Then, we can deal with the unique solution $u \in W^{1,2}_0 (\Omega)$ that exists by \Cref{thm:Boccardo well-posedness}. Due to the construction by approximation in \Cref{thm:existence E in L^2} the estimates pass to the limit in the construction.
		Take $h \in W^{1,\infty}(\mathbb R)$ such that $h(0) = 0$. We take  $v = h(u)$ as a test function we can write
		\begin{equation*}
			\alpha \int_\Omega h'(u) |\nabla u|^2 \le \int_\Omega M(x) \nabla u \cdot \nabla h(u) =    \int_\Omega u E \cdot \nabla h(u)  + \int_\Omega f h(u).
		\end{equation*}
		We can write
		\begin{equation*}
			u \nabla h(u) = u h'(u) \nabla u = \nabla F(u)
		\end{equation*}
		where $F(s) = \int_0^s \tau\,h'(\tau)  d\tau$. Hence,
		\begin{equation*}
			\alpha \int_\Omega h'(u) |\nabla u|^2 \le  \int_\Omega  E \cdot \nabla F(u)  + \int_\Omega f h(u).
		\end{equation*}
		Now we prove both items
		\begin{itemize}
			\item \textit{{\Cref{it:sum u in L1}}.}
			{Since $\diver E \in L^1 (\Omega)$ we can 
			integrate by parts again to deduce
			\begin{equation}
				\label{26mar}
				\alpha \int_\Omega h'(u) |\nabla u|^2 +\int_\Omega  F(u) \diver E \le \int_\Omega f h(u).
			\end{equation}
			}
			Let us consider
			$h_\varepsilon (s) = T_\varepsilon(s)/\varepsilon$. Then $h_\varepsilon' \ge 0$ and $|h_\varepsilon| \le 1$ and, hence,
			in \eqref{26mar}
			\begin{equation*}
				\int_\Omega  F_\varepsilon (u) \diver E \le \int_\Omega |f|.
			\end{equation*}
			It is clear that $F_\varepsilon(s) \to |s|$ a.e.\ as $\varepsilon \to 0$. Then, 
			\begin{equation*}
				\int_\Omega |u| \diver E \le \int_\Omega |f|.
			\end{equation*}
			
			\item \textit{\Cref{it:sum u in m greater 1}. } Let us take, for $m > 1$, $h(s) = |s|^{m-1}$ then 
			\begin{equation*}
				F(s) = (m-1)\int_0^s |\tau|^{m-2} \sign (\tau) \tau d\tau = \frac{m-1}{m} s^m .
			\end{equation*}
			Hence, going back to \eqref{26mar}
			\begin{equation*}
				c_0 \frac{m-1}{m} \| u \|_{L^m}^m \le  \frac{m-1}{m} \int_\Omega |u|^m  \diver E \le \int_\Omega f |u|^{m-1} \le \| f \|_{L^m} \| u \|_{L^m}^{m-1}.
			\end{equation*}
			Hence, we simplify
			\begin{equation*}
				\| u \|_{L^m} \le \frac{m}{m-1} \frac{\| f \|_{L^m} }{c_0} . \qedhere
			\end{equation*}
		\end{itemize}
	\end{proof}

\section{Comparison principle and {uniqueness}}

\label{sec:comparison principle}

To show uniqueness of solutions we prove a weak maximum principle.

\begin{theorem}%
	\label{thm:uniqueness}
	Let $f \in L^{\frac{2N}{N+2}} (\Omega)$ and 
	\eqref{e}.  
	Then, if $u \in W^{1,2}_0(\Omega)$ is a solution of \eqref{eq:weak solution} then 
	\begin{equation*}
		\| \nabla u^+ \|_2 \le \frac{1}{\alpha S_2} \| f^+ \|_{\frac{2N}{N+2}}  .
	\end{equation*} 
	Hence, there is, at most, one solution of \eqref{eq:weak solution} {in $W_0^{1,2} (\Omega)$}. Furthermore, if $f \ge 0$ then $u \ge 0$.
\end{theorem}	

We first prove the following lemma
\begin{lemma}
	\label{lem:divergence non-negative}
	Let $m, {r} > 1$, $E \in L^{r'} (\Omega)$
	 with $0 \le \diver E \in \mathcal D'(\Omega)$. Then, we have that
	\begin{equation}
		-\int_\Omega E \nabla v \ge 0 \qquad \forall\; 0 \le v \in W^{1,r}_0 (\Omega).
	\end{equation}	
\end{lemma}
\begin{proof}
	By definition of having a sign in distributional sense, for $0 \le \varphi  \in \mathcal C_c^\infty(\Omega) $, we have that
	\begin{equation*}
		- \int_\Omega E \nabla \varphi = \langle \diver E , \varphi \rangle \ge 0. 
	\end{equation*}
	For $0 \le v \in W^{1,r}_0 (\Omega)$, we can find a sequence $0 \le \varphi_n \in \mathcal C_c^\infty(\Omega)$, {such that} $\varphi_n \to v $ in $W^{1,r}_0 (\Omega)$. In particular, $\nabla \varphi \to \nabla v$ in $L^r (\Omega)$. We can pass to the limit in the estimate.
\end{proof}

\begin{proof}[Proof of \Cref{thm:uniqueness}]
	Let $u$ be a solution. {Take $\rho_n$ a family of non-negative mollifiers,} and use $v{_n} = {\rho_n*}u^+$ as a test function. 
	{Passing to the limit in $n$ and}
	applying the  
	previous 
	lemma
	\begin{equation*}
		\alpha \int_\Omega |\nabla u^+|^2 \le \int_\Omega E \nabla \frac{u_+^2}{2} + \int_\Omega f u^+ \le \| f \|_{(2^*)'} \| u^+ \|_{2^*} \le \frac{1}{S} \| f \| _{(2^*)'} \| \nabla u^+ \|_2. 
	\end{equation*}
	We recover the estimate.
\end{proof}

{
\begin{lemma}
	\label{rem:approximation E} 
	Let $E \in L^r (\Omega)^N$ for $r > 1$ with $\diver E \ge 0$ in $D'(\Omega)$. Then, there exists a sequence $E_n \in W^{1,\infty} (\Omega)$ with $\diver E_n \ge 0$ such that $E \to E$ in $L^r(\Omega)^N$.
\end{lemma}
\begin{proof}
	We use a similar decomposition to \cite[Theorem 1.5]{RogerTemam1979} (done there for $r = 2$).
	First, define 
	$$
	\begin{dcases}
		-\Delta p^{(1)} = \diver E & \text{in }\Omega,\\
		p^{(1)} = 0 & \text{on }\partial \Omega.
	\end{dcases}
	$$
	By well-known results we get a unique solution $p^{(1)} \in W^{1,r}_0 (\Omega)$. 
	Take $E^{(1)} = \nabla p^{(1)}$.
	Lastly, take $E^{(2)} = E - E^{(1)} \in L^r (\Omega)$.
	Notice that $\diver E^{(2)} = 0$.
	Due to \cite{Kato2000}, $E^{(2)}$ admits a divergence-free extension to $L^r(\Rd)$, which we denote $\widetilde E^{(2)}$. We can take a family of mollifiers $\rho_n$, and $E_n^{(2)} = \widetilde E^{(2)} * \rho_n \in W^{1,\infty} (\Omega)$.
	Now let $0 \le g^{(1)} = \diver E^{(1)} \in W^{-1,r'} (\Omega)$. Let $g_n^{(1)} \ge 0$ be a sequence of $C_c^\infty (\Omega)$ functions with $g_n^{(1)} \to g^{(1)}$ in $L^{r'} (\Omega)$. Take $p_n^{(1)}$ the unique solution to
	$$
	 \begin{dcases}
		-\Delta p^{(1)}_n = g_n^{(1)} & \text{in }\Omega,\\
		p^{(1)}_n = 0 & \text{in }\partial \Omega.\\
	 \end{dcases}
	$$
	Finally, define $E_n^{(1)} = \nabla p_n^{(1)} \in W^{1,\infty} (\Omega).$
	It is now easy to see that $E^{(i)}_n \to E^{(i)}$ in $L^r (\Omega)^N$ for $i=1,2$, and the proof is complete.
\end{proof}
}

\begin{theorem}
	\label{thm:existence and uniqueness}
	Let $f\in L^m (\Omega)$ and $E \in L^r (\Omega)$ such that $0 \le \diver E \in \mathcal D'(\Omega)$ and
	\begin{equation}
		\label{eq:condition existence weak minus}
		\begin{dcases}
					\frac{1}{ \min \{2^* , m^{**}\} } + \frac{1}{r} \le  1  & \text{if } \tfrac {{2}N}{N+2} \le m \le \tfrac N 2 \\
					\frac{1}{ 2^* } + \frac{1}{r} \le  1  & \text{otherwise} \\
		\end{dcases}
	\end{equation}
	Then, taking $q = \min \{ 2 , m^* \}$ 
	{(using formally $m^* = \infty$ for $m \ge N$)} there exists a solution of 
	\begin{equation}
		\label{eq:weak minus}
		u \in W^{1,q  }_0 (\Omega) \quad \text{such that} \quad  \int_\Omega M(x) \nabla u \nabla v = \int_\Omega u E \nabla v + \int_\Omega f v, \quad \forall v \in W^{1,q'}_0 (\Omega).
	\end{equation}
	Furthermore, 
	if $m \ge \frac{2N}{N+2}$ and $r \ge N$
	it is the unique solution of \eqref{eq:weak solution}.
\end{theorem}

\begin{proof}
	Let $f_k = T_k (f)$ where $T_k$ is the cut-off function. We consider $E_k$ constructed in 
		\Cref{rem:approximation E}.
	\\
	By \Cref{thm:a priori diver E non-neg} there exists a unique weak $u_k$ solution of \eqref{eq:weak solution}. Since the $\cdot^*$ operation is monotone, then $q^* = \min \{ 2^*, m^{**}  \}$.  The sequence $u_k$ is uniformly bounded in $ W_0^{1,q} (\Omega)$. Therefore, by the Sobolev embedding theorem, it is uniformly bounded on $L^{ q^* } (\Omega)$.
	\\	
	Up to a subsequence, there exists $u \in W_0^{1,q} (\Omega)$ such that
	\begin{align*}
		\nabla u_k \rightharpoonup  \nabla u & \qquad \text{in }  L^{q} (\Omega) \\
		u_k \rightharpoonup u & \qquad \text{in }  L^{q^*} (\Omega).
	\end{align*}
	Since $M \in L^\infty (\Omega)^{N\times N}$, $E_k \to E \in L^r(\Omega)^N$ strongly  and \eqref{eq:condition existence weak minus} we have that
	\begin{align*}
		M(x) \nabla u_k \rightharpoonup M(x)  \nabla u & \qquad \text{in }  L^{q} (\Omega) \\
		u_k E_k \rightharpoonup u E & \qquad \text{in }  L^1 (\Omega).
	\end{align*}
	Therefore, we can pass to the limit in the weak formulation for $v \in W^{1,\infty}_0 (\Omega)$.
	If $m \ge \frac{2N}{N+2}$ and $r \ge N$, then $u E \in L^{2} (\Omega)$, and it is a solution of \eqref{eq:weak solution} by approximation.
\end{proof}

\section{Convection with singularity at one point}
\label{sec:convection singular one point}

With the approach developed in this paper we are able to study the special situation
\begin{equation} 
	\label{eq:E depending on A}
	E = A \frac x {|x|^2} \qquad \text{ where } A > 0
\end{equation} 
which is somehow in the limit of theory since it is not in $L^N(\Omega)$, but it is in $L^{r} (\Omega)$ for $r \in [1,N)$.
In \cite{Boccardo+Orsina2019} the authors examined the framework of drifts such that 
\begin{equation} 
	\label{eq:|E|le|A|/|x|}
	|E|\le\frac{|A|}{|x|},
\end{equation}
The authors show existence of solutions $u$ under \eqref{eq:|E|le|A|/|x|}, where the summability is reduced as $|A|$ is increased. They prove 
\begin{theorem}[\cite{Boccardo+Orsina2019}]
	Let $f \in  L^m (\Omega)$ and $|E| \le |A| / |x|$. Then, there exists a solution $u$ the solution of \eqref{eq:PDE} and
	\begin{enumerate}
		\item If $|A| < \frac{\alpha ( N - 2m )}{m}$ and $m \in [ \frac{2N}{N+2}, \frac N 2)$ then $u \in W_0^{1,2} (\Omega) \cap L^{m^{**}} (\Omega)$.
		\item If $|A| < \frac{\alpha ( N - 2m )}{m}$ and $m \in (1, \frac{2N}{N+2} )$ then $u \in W_0^{1,m^*} (\Omega)$.
		\item If $|A| < \alpha (N-2)$ and $m = 1$ then $\nabla u \in (M^{\frac{N}{N-1}} (\Omega))^N$ and $u \in W_0^{1,q} (\Omega)$, for every $ q < \frac {N}{N-1}$.
	\end{enumerate}
\end{theorem}
Above, $M^{\frac{N}{N-1}}$ denotes the Marcinkiewicz space (see \cite{Boccardo+Orsina2019} for the definition and some properties).  
The argument in \cite{Boccardo+Orsina2019} is based on Hardy's inequality
\begin{equation}
	\label{eq:Hardy}
	\left( \frac{N-2}N \right)^2 \int_{\mathbb R^N} \frac{|u|^2}{|x|^2 } \le \int_{\mathbb R^N} |\nabla u|^2 .
\end{equation}
We are able to extend this result to distinguish depending on the sign of $A$. Our result is the following
\begin{theorem}
	\label{prop:convection at one point}
	Let
	 $f \in L^m (\Omega)$ for some $m > 1$ and \eqref{eq:E depending on A}. 
	Then, there exists a solution $u_A$ of \eqref{eq:weak minus}, and it satisfies the estimates in \Cref{thm:a priori diver E non-neg}. Furthermore, $u_A \to 0$ as $A \to \infty$ in the sense that
		\begin{equation*}
		\int_\Omega \frac{|u_A(x)|}{|x|^2} \le \frac{1}{A (N-2)} \int_\Omega |f|.
	\end{equation*}
\end{theorem} 
We point out that,  if $m > \frac{2N}{N+2}$, we have furthermore $u{_A}E \in L^2(\Omega)$. 
\begin{proof} 
	Since $N \ge 3$ we know that $|E| \in L^2 (\Omega)$ and that 
	\begin{equation}
		\label{eq:E singular 0 divergence}
		\diver E = r^{1-N} \frac{\partial}{\partial r} (  r^{N-1} A r^{-1}  ) =\frac{ A(N-2)}{|x|^2} 
	\end{equation}
	is non-negative, and it is in $L^1 (\Omega)$. Then, we have satisfied the existence theory of \Cref{thm:existence E in L^2}. Due to \Cref{thm:a priori diver E non-neg} and \eqref{eq:E singular 0 divergence} the estimate follows.
\end{proof}

\section{Convection with singularity on the boundary}

\label{sec:convection singular boundary}

The aim of this section is to understand the case where $E$ is regular inside $\Omega$ but blows up towards $\partial \Omega$.
For the sake of simplicity we present an example, which as mentioned in \Cref{sec:extension} can be generalized, but the computations become quite technical.
Let us consider $\varphi_1$ the first eigenfunction of $-\Delta$ with Dirichlet boundary conditions, i.e.,
\begin{equation*}
	\begin{dcases}
		-\Delta \varphi_1 = \lambda_1 \varphi_1 & \text{in } \Omega, \\ 
		\varphi_1 = 0 & \text{on } \partial \Omega.
	\end{dcases}
\end{equation*}
We normalize it so that $\|\nabla \varphi_1 \|_{L^\infty} = 1$. It is known that there exists $C > 0$ such that
\begin{equation*}
	0 < C \dist(x,\partial \Omega) \le \varphi_1 (x) \le C^{-1} \dist(x,\partial \Omega), \qquad \forall x \in \Omega,
\end{equation*}
and near $\partial \Omega$ we have that
$$	| \nabla \varphi_1 (x)| \ge C > 0.$$
We focus our efforts on the particular case
\begin{equation}
	\label{E:superlinear}
	E = - \varphi_1^{-1-\gamma} \nabla \varphi_1 , \qquad \text{ for some } \gamma > 0,
\end{equation} 
and 
$f \in L^\infty_c (\Omega)$,  
the space  
of bounded functions  with compact support in $\Omega$.
The aim of this section is to prove
\begin{theorem}
	\label{thm:E singular boundary}
	Let $E$ be given by \eqref{E:superlinear}, $M = I$ and $f \in L_c^\infty (\Omega)$. 
	Then, there exists a unique $u \in H^1_0 (\Omega) \cap L^\infty (\Omega)$ such that $u E \in L^\infty (\Omega)$ and $u$ is a weak solution in the sense that \eqref{eq:weak solution} {holds}. Furthermore, $u$ is flat on the boundary in the sense that 
	\begin{equation}
	\text{for all } \alpha > 1~ \text{ we have that } \qquad 	|u (x) |\le C_\alpha \dist(x, \partial \Omega)^\alpha , \qquad \text{for a.e. } x \in \Omega.
	\end{equation}
\end{theorem}

{ 
	We will give the proof below. First, we prove positivity in the interior.
\begin{proposition}
	In the assumptions of \Cref{thm:E singular boundary} if $f \ge 0$ and $\int_\Omega f > 0$, then $u > 0$ in $\Omega$.
\end{proposition}

\begin{proof}
	Let $\Omega_\eta = \{ x \in \Omega : d(x,\Omega) > \eta  \}$. Consider $u_\eta$ the solution of \eqref{eq:PDE} with $E$ given by \eqref{E:superlinear} and $u_\eta = 0$ in $\partial \Omega_\eta$. 
	Notice that $E$ is smooth on $\Omega_\eta$ for $\eta > 0$. 
	Since we already know from \Cref{thm:uniqueness} that $u \ge 0$ in $\Omega$, the classical comparison principle in $\Omega_\eta$ ensures that $u_\eta \le u$ for any $\eta \ge 0$. 
	Take $\eta > 0$ small enough so that $\int_{\Omega_\eta} f > 0$. 
	Then, by the ``classical'' strong maximum principle we get $u_\eta > 0$ in $\Omega_\eta$, and the proof is complete.
\end{proof}
}

 It is immediate to compute that
\begin{align*}
	\diver E &= (1 + \gamma) \varphi_1^{-2-\gamma} |\nabla \varphi_1|^2 -  \varphi_1^{-1-\gamma} \Delta  \varphi_1  \\
	&=  (1 + \gamma) \varphi_1^{-2-\gamma} |\nabla \varphi_1|^2 + \lambda_1  \varphi_1^{-\gamma} 
\end{align*}
Hence $	\diver E (x) \ge c \dist(x, \partial \Omega)^{-2-\gamma}$ near the boundary. Notice that $E$ and $\diver E$ are not in $L^1 (\Omega)$. We start the proof with a lemma.

\begin{lemma}
	\label{thm:existence singular boundary}
	In the assumptions of \Cref{thm:existence and uniqueness}, assume furthermore that $\diver E \in L^1_{loc} (\Omega)$. Then $u \diver E \in L^1 (\Omega)$, still satisfying estimate \eqref{eq:uniform bound 2}.
\end{lemma}

\begin{proof}%
	We consider the approximating sequence for \Cref{thm:existence and uniqueness}. For the approximation we know that
		\begin{equation*}
		\int_\Omega |u_n| \diver E_n \le \int_\Omega |f|.
		\end{equation*}
		Let us fix $K \Subset \Omega$. We have that
		\begin{equation*}
		\int_K   |u_n| \diver E_n \le \int_\Omega |f| .
		\end{equation*}
		Since we know that $\diver E_{n} \to \diver E$ in $L^1 (K)$, we have that, up to a further subsequence, $\diver E_n$ converges a.e. in $K$. Hence, applying Fatou's lemma
		\begin{equation*}
		\int_K   |u| \diver E \le \int_\Omega |f| .
		\end{equation*}
		Since this estimate is uniform in $K$, we can take $ K_h = \{ x \in \Omega : \dist(x, \partial \Omega) \ge h \}$ and deduce, as $h \to 0$, that \eqref{eq:uniform bound 2} holds. 
	\end{proof}
The solution found in \Cref{thm:E singular boundary}  
is unique in a certain class. We provide a uniqueness result extending \Cref{thm:uniqueness}, which  
can itself be generalised to a larger framework 
\begin{lemma}
	\label{lem:E blow up uniqueness}
	Assume that $u \in H_0^1 (\Omega)$, $E \in L^\infty_{loc} (\Omega)$, $u |E| \in L^2 (\Omega)$, $\diver E \ge 0$ distributionally, and $f \in L^{\frac{2N}{N+2}} (\Omega)$. Then
	\begin{equation*}
		\| \nabla u^+ \|_2 \le \frac{1}{\alpha S_2} \| f^+ \|_{\frac{2N}{N+2}}
	\end{equation*}
	In particular, there is at most one weak solution in $H_0^1 (\Omega)$ of the \eqref{eq:PDE}.
\end{lemma}	

\begin{proof}
	We want to repeat the argument in \Cref{thm:uniqueness}, i.e., taking $v = u_+$ in the weak formulation and using that
	\begin{equation*}
		-\int_\Omega u E \cdot \nabla u_+ \ge 0.
	\end{equation*}
	We prove this formula by approximation.
	Take $\eta \in C_c^\infty (\Omega)$. 
	There exists $K \Subset \Omega$ and $\phi_m \in C_0^\infty(K)$ such that $\phi_m \to u_+ \eta $ in $H_0^1(\Omega)$. We have that
	\begin{equation*}
		- \int_\Omega \phi_m E  \cdot \nabla \phi_m =\left  \langle \diver E , \frac{\phi_m^2}{2} \right \rangle \ge 0.
	\end{equation*}
	Since $E \in L^\infty (K)$, we pass to the limit to deduce
	\begin{equation*}
		- \int_\Omega (u_+ \eta) \cdot E \nabla (u_+ \eta)  \ge 0.
	\end{equation*}
	Now we expand
	\begin{equation*}
		\int_\Omega (u_+ \eta)  E  \cdot \nabla (u_+ \eta)  = \int_\Omega u_+^2 \eta E  \cdot \nabla \eta + \int_\Omega u_+ \eta^2 E  \cdot \nabla u_+
	\end{equation*}
	Now we take $\eta_m \nearrow 1$. In particular $\eta_m (x) = \eta_0 ( m \varphi_1(x)   )$ where $\eta_0$ is non-decreasing, $\eta_0(s) = 0$ if $s\le 1$ and $\eta_0(s) = 1$ if $s > 2$. 
	Clearly $\|\nabla \eta_m\|_{L^\infty} \le Cm$.
	Since $u_+ \in H_0^1(\Omega)$, then $u_+(x) / \varphi_1(x) \in L^2 (\Omega)$ by Hardy's inequality. And we compute
	\begin{equation*}
		\left| \int_\Omega u_+^2 \eta_m \cdot E \nabla \eta_m \right| \le \int_{ \varphi_1(x) \le \frac 1 m  } \frac {u_+} {\varphi_1} \frac{\varphi_1 }{m} |u E| C  m \le C \int_{ \varphi_1(x) \le \frac 1 m  } \frac {u_+} {\varphi_1} |u E|  \to 0
	\end{equation*}
	since $\frac {u_+} {\varphi_1}|uE| \in L^1(\Omega)$ and the size of the domain tends to zero. We conclude, by Dominated Convergence that
	\begin{equation*}
		0 \ge  \int_\Omega (u_+ \eta_m ) \cdot E \nabla (u_+ \eta_m ) \to \int_\Omega u_+ E \cdot \nabla u_+ = \int_\Omega u E \cdot u_+. \qedhere 
	\end{equation*} 
\end{proof}

We are finally ready to prove the result.

\begin{proof}[\textbf{Proof of \Cref{thm:E singular boundary}}]
	The uniqueness claim is proven in \Cref{lem:E blow up uniqueness}. We now prove the existence and bounds by approximation.
	We can assume, without loss of generality, that $f \ge 0$, and construct approximations
	of $E$ given by
	\begin{equation*}
		E_\ell = - ( \varphi_1 + \tfrac 1 \ell )^{-1-\gamma} \nabla \varphi_1 .
	\end{equation*}
	Clearly $E_\ell \in L^\infty (\Omega)$. 
	These {satisfy the assumptions of} \Cref{thm:existence E in L^2}. 
	Hence, there exists a weak solution $u_\ell \in H^1 _0 (\Omega)$ {of \eqref{eq:PDE} where $E = E_\ell$}. 
	We compute
	\begin{align*}
		\diver E_\ell
		&=  (1 + \gamma)( \varphi_1 + \tfrac 1 \ell )^{-2-\gamma}  |\nabla \varphi_1|^2 + \lambda_1  ( \varphi_1 + \tfrac 1 \ell )^{-1-\gamma} \varphi_1 .
	\end{align*}
	This is non-negative. Hence, due to \Cref{thm:uniqueness} we have that
	\begin{equation*}
		\|\nabla u_\ell \|_{L^2} \le C \|f \|_{L^\infty} .
	\end{equation*}
	Splitting the behaviour near the boundary and away from the boundary, it is easy to see that $\diver E_\ell \ge c_0 > 0$ uniformly. 
	Therefore, due to \Cref{thm:a priori diver E non-neg} we have that
	\begin{equation}
		\label{eq:E singular boundary Linfty 1}
		\|u_\ell \|_{L^\infty} \le \frac{\| f \|_{L^\infty} }{c_0}.
	\end{equation}
	Now we must construct barrier functions. Select a single $\alpha > 1$ and the barrier
	\begin{equation*}
		U = \tfrac 1 \alpha (\varphi_1 + \tfrac 1 \ell)^\alpha. 
	\end{equation*}
	We drop the dependence on $\ell$ and $\alpha$ to make the presentation below more readable.
	Plugging it into the equation we get
	\begin{align*}
		-\Delta U + \diver (U E_\ell) 
		&=  -\Delta U + \nabla U \cdot E_\ell + U \diver E_\ell \\
		&= -(\alpha - 1)( \varphi_1 + \tfrac 1 \ell )^{\alpha - 2}  |\nabla \varphi_1|^2 + \lambda_1  ( \varphi_1 + \tfrac 1 \ell )^{\alpha - 1} \varphi_1  \\
		&\qquad - ( \varphi_1 + \tfrac 1 \ell )^{\alpha - 2 - \gamma }  |\nabla \varphi_1|^2 \\
		&\qquad + (1 + \gamma)( \varphi_1 + \tfrac 1 \ell )^{\alpha -2-\gamma}  |\nabla \varphi_1|^2 + \lambda_1  ( \varphi_1 + \tfrac 1 \ell )^{\alpha -1-\gamma} \varphi_1 \\
		&\ge  \Big( \gamma ( \varphi_1 + \tfrac 1 \ell )^{-\gamma} -(\alpha - 1) \Big) ( \varphi_1 + \tfrac 1 \ell )^{\alpha - 2}  |\nabla \varphi_1|^2 .
	\end{align*}
	This is non-negative if
	$ \varphi_1 (x) + \frac 1 {m} \le  ( \tfrac{ \alpha - 1}{\gamma} )^{-\frac 1 \gamma}$.
	There exists $\eta_\alpha > 0$ small enough such that 
	\begin{equation*}
		 f(x) = 0  \qquad \text{and} \qquad  \varphi_1 (x)  \le \tfrac 1 2   ( \tfrac{ \alpha - 1}{\gamma} )^{-\frac 1 \gamma}, \qquad  \forall x \text{ such that } \dist(x, \partial \Omega) \le  \eta_\alpha.
	\end{equation*}
	We will use the neighbourhood of the boundary
	$	A_ \alpha  = \{ x \in \Omega: \dist(x, \partial \Omega)  < \eta_\alpha   \} $.
	Also, we consider the candidate super-solution
	\begin{equation*}
		\overline {u} (x) =  U (x) \left(  \dfrac{ \alpha }{   \min_{ \dist(x, \partial \Omega) = \eta_\alpha }  \varphi_1 (x)^\alpha  }   + \frac \alpha {c_0 }  \frac{  \| f \|_{L^\infty} }{  \min_{\dist(x, \partial \Omega) \ge \eta_\alpha}     \varphi_1 (x) ^\alpha      }    \right) .
	\end{equation*}
	We denote the constant on the right-hand side as $C_\alpha$.
	Using the first term of $C_\alpha$, 
	 $\overline u \ge u$ when $\dist(x, \partial \Omega) = \eta_\alpha$. Also, $\overline u = \frac 1 {m^\alpha} \ge 0 = u$ on $\partial \Omega$. 
	 Let us call
	\begin{equation*}
		\overline f =	-\Delta \overline u + \diver (\overline u E_\ell) .
	\end{equation*}
	By the 
	previous  
	computations, 
	if $\ell \ge 2 ( \tfrac{ \alpha - 1}{\gamma} )^{\frac 1 \gamma}$,
	we have $\overline f \ge 0 = f$ in $A_\alpha$, and clearly $\overline f \in L^\infty (A_\alpha)$. 
	Hence, due to \Cref{thm:uniqueness} we have that
	\begin{equation*}
		0 \le u_\ell (x) \le \overline {u} (x), \qquad  x \in A_\alpha .
	\end{equation*}
	Also, due to \eqref{eq:E singular boundary Linfty 1} and the second part of $C_\alpha$, we have that
	\begin{equation*}
		0 \le u_\ell (x) \le \overline {u} (x), \qquad  x \in \Omega \setminus A_\alpha .
	\end{equation*}
	Eventually, we deduce that for any $\alpha > 1$ we that
	\begin{equation*}
		0 \le u_\ell(x)  \le \tfrac {C_\alpha} \alpha (\varphi_1 + \tfrac 1 \ell)^\alpha , \qquad \forall x \in \Omega \text{ and } \ell \ge 2 ( \tfrac{ \alpha - 1}{\gamma} )^{\frac 1 \gamma}.
	\end{equation*}
	In particular, picking $\alpha = \gamma + 1$ we deduce that 
	\begin{equation*}
		|u_\ell E_\ell | \le   \tfrac {C_{\gamma + 1}} {\gamma + 1}  \|\nabla \varphi_1\|_{L^\infty} = \tfrac {C_{\gamma + 1}} {\gamma + 1} .
	\end{equation*}
	We deduce that, up to a subsequence,
	\begin{equation*}
		u_\ell \to u \text{ a.e. and strongly in } L^2 \qquad \text{ and }  \qquad u_\ell \rightharpoonup u \text{ weakly in } H_0^1 (\Omega).
	\end{equation*}
	This implies that $u_\ell E_\ell \to u E$ a.e. And hence $u E$ is bounded. Passing to the limit in the weak formulation by the Dominated Convergence Theorem, the result is proven.
\end{proof}

\begin{remark}
	Notice that the construction of the super-solution above can be done in any dimension $N \ge 1$. However, most of the results in the rest of the paper are only available for $N \ge 3$. 
\end{remark}

\begin{remark}
	\label{rem:19 de Ildefonso}
	For Schrödinger-type equations $-\Delta u + Vu = f$, it is known that if the potential $V$ is greater than $\dist(x,\partial \Omega)^{-2}$ and $f$ is compactly supported, then $u$ is flat {on} the boundary, in the sense that $|u| \le C\dist(x,\partial \Omega)^{1+\varepsilon}$. This means that $\partial_n u = 0$ on $\partial \Omega$. This means that it satisfies Dirichlet and Neumann homogeneous boundary conditions. And it can be extended by $0$ outside $\Omega$ with higher regularity than $H^1$.
	In contrast, the exponent $\gamma $ in the above result can
	not be taken as $\gamma =0$ in order to get flat solutions. Indeed, the
	convection term $E\cdot \nabla \varphi _{1}$, in the above computations, is
	more singular than the term $\varphi _{1} \diver E$. A very explicit example
	can be done in one dimension:  if we consider $E=-Cx^{-1}$ then this
	drift does not generate flat solutions since if we take $U(x)=x^{\alpha }$
	then 
	\[
	-U^{\prime \prime }+(EU)^{\prime }=(-\alpha x^{\alpha -1}-Cx^{\alpha
		-1})'=-(\alpha +C)(\alpha -1)x^{\alpha -2},
	\]%
	and this is a supersolution only if $\alpha \leq 1.$
\end{remark}

\begin{corollary}
	In the hypothesis of \Cref{thm:E singular boundary} replace $f \in L^\infty_c (\Omega)$ by 
	$$	
		|f(x)| \le C \dist( x ,\partial \Omega)^{\omega} \qquad \text{ for } {0 \le \omega \le \gamma + 1}.
	$$
	Then
	$$|u(x)| \le \dist( x ,\partial \Omega)^{\alpha} \qquad  \text{for all } \alpha \in \left (1 , \gamma + 2 - \omega \right). $$
\end{corollary}
\begin{proof}
	We maintain the notation of the proof of \Cref{thm:E singular boundary}.
	We have already shown that, on a neighbourhood of the boundary,
	$$
	-\Delta U + \diver (U E_m) \ge \frac{\gamma}{2} (\varphi_1 + \tfrac 1 m)^{\alpha - 2 - \gamma } |\nabla \varphi_1|^2 \ge c_1 (\varphi_1 + \tfrac 1 m)^{\alpha - 2 - \gamma }
	\ge c_2 |f|.
	$$
	For $\alpha$ in the range $\left (1 , \gamma + 2 - \omega \right)$, we can take as a supersolutions for the approximating sequence
	\begin{equation*}
		\overline {u} (x) =  U (x) \left(  \frac{1}{c_2} +   \dfrac{ \alpha }{   \min_{ \dist(x, \partial \Omega) = \eta_\alpha }  \varphi_1 (x)^\alpha  }   + \frac \alpha {c_0 }  \frac{  \| f \|_{L^\infty} }{  \min_{\dist(x, \partial \Omega) \ge \eta_\alpha}     \varphi_1 (x) ^\alpha      }     \right) . \qedhere 
	\end{equation*}
	And the rest of the proof remains as in \Cref{thm:E singular boundary}.
\end{proof}

\section{Further remarks, extensions, and open problems}
\label{sec:extension}

\begin{enumerate}
	\item We point that the proofs of our estimates 
	{can be extended to}
	many non-linear settings. 
	
	\item \Cref{thm:E singular boundary} admits many generalisations. 
	For instance, one can consider
	the case $|E| \le c_0 \dist(x, \partial \Omega)^{-\gamma-1}$ with $\diver E \ge c_1 \dist(x, \partial \Omega)^{-\gamma -2}$ up to suitable conditions on the constants. Also,
	the techniques in this paper could be extended to 
	the situation where $ \dist(x, \partial \Omega)$ is replaced by $\dist(x, \Gamma)$ with a suitable part $\Gamma \subset \partial \Omega$. The case $\Gamma$ an interior manifold can also be studied.

	\item Including a non-negative potential. The same analysis can be performed on the equation
	\begin{equation*}
		-\diver (M(x) \nabla u) + a(x) u = - \diver(u E(x)) + f(x)
	\end{equation*}
	when $a \ge 0$. 
	{As above, our approach allows for less regularity in $a$ than most previous literature, e.g. $a \in L^1_{loc} (\Omega)$.}
	Furthermore, one will then obtain 
	\begin{equation*}
		\int_ \Omega |u| (a + \diver E ) \le \int_\Omega |f|.
	\end{equation*} 
	Hence, one can reduce the hypothesis to $a + \diver E \ge 0$ in the whole analysis.

	\item 
	The study of $a \equiv 1$ is useful in the study of the evolution problem
	\begin{equation*}
			u_t -\diver (M(x) \nabla u) +  \diver(u E(x)) = 0.
	\end{equation*} 
	For the study of this problem one can write $u_t + Au = 0$ 
	where
	\begin{equation*}
		A u =  -\diver (M(x) \nabla u) +  \diver(u E(x)).
	\end{equation*} 
	In order to obtain solutions in semigroup form in $L^p$ (where $ 1 \le p \le +\infty$), 
	following the theory of accretive operators, it is sufficient that,
	\begin{equation*}
		\| u \|_{L^p} \le \| u + \lambda A u \|_{L^p}.
	\end{equation*}
	Letting $f = u + \lambda A u$, this is precisely what we have proven above, where $M = \lambda I$ and $a \equiv 1$. See also \cite{bop}.
	
	\item 
	We point out that when $|E| \le |A| / |x|$, we have that, if $m>\frac{2N}{N+2}$ then $u|E| \in L^2 (\Omega)$. It seems possible to extend the uniqueness result \eqref{lem:E blow up uniqueness} to this setting.
\end{enumerate}

\section*{Acknowledgements}

This paper was started during the doctoral course of LB in Madrid in 2019, funded by the Instituto de Matemática Interdisciplinar.
The research of J. I. Díaz was partially supported by the project PID2020-112517GB-I00 of the DGISPI (Spain) and the Research Group MOMAT (Ref. 910480) of the UCM.
The research of DGC was supported by the Advanced Grant Nonlocal-CPD (Nonlocal PDEs for Complex Particle Dynamics:
Phase Transitions, Patterns and Synchronization) of the European Research Council Executive Agency (ERC) under the European Union’s Horizon 2020 research and innovation programme (grant agreement No. 883363).

\end{document}